\documentclass{amsart}

\usepackage{amssymb,amsmath,latexsym,amsthm,amsrefs,breqn}

\newtheorem{theorem}{\indent Theorem}[section]
\newtheorem{definition}[theorem]{\indent Definition}

\newtheorem{conj}[theorem]{\indent Conjecture}

\newcommand{\Gal}{\mathrm{Gal}}
\newcommand{\fa}{-8ax^3-(8a+2)x^2+(4a-1)x+a}
\newcommand{\fafa}{4096a^4x^9 + (12288a^4 + 3072a^3)x^8 + (6144a^4 + 7680a^3 + 768a^2)x^7\\ + (-9728a^4 + 2560a^3 + 1408a^2 + 64a)x^6 + (-6144a^4 - 4864a^3 + 64a^2 + 32a)x^5\\ + (3072a^4 - 1920a^3 - 672a^2 - 80a - 8)x^4 + (1216a^4 + 1024a^3 - 64a^2 - 16a - 8)x^3\\ + (-192a^4 + 240a^3 + 40a^2 + 16a)x^2 + (-96a^4 - 40a^3 + 16a^2 - 4a + 1)x\\ + (-8a^4 - 8a^3 + 2a^2)}

\newcommand{\fafaS}{128a^2x^6+(256a^2+32a)x^5+ 32ax^4 + (-160a^2-16a-4)x^3\\-(4a+2)x^2+(16a^2+1)x+2a^2}

\begin{document}
\today

\title{An Infinite Family of Cubics with Emergent Reducibility at Depth 1}
\author{Jason I.~ Preszler}
\email{jpreszler@member.ams.org}

\begin{abstract}
A polynomial $f(x)$ has emergent reducibility at depth $n$ if $f^{\circ k}(x)$ is irreducible for $0\leq k \leq n-1$ but $f^{\circ n}(x)$ is reducible. In this paper we prove that there are infinitely many irreducible cubics $f \in \mathbb{Z}[x]$ with $f\circ f$ reducible by exhibiting a one parameter family with this property. 
\end{abstract}

\maketitle

\section{Introduction}
Given a polynomial $f(x) \in \mathbb{Q}[x]$, one can construct the sequence of iterates $f\circ f(x), f\circ f\circ f(x), \dots, f^{\circ n}(x)$. Such sequences form dynamical systems and have become the focus of considerable scrutiny in recent years. In the 1980's, Odoni proved fundamental facts about the behavior of the discriminant and resultant \cite{odoniPlms}, proved instances where entire sequences consisted of irreducible polynomials \cite{odoniJlms}, gave examples of sequences with irreducible initial terms but reducible terms after a certain point \cite{odoniJlms}, and showed that the Galois groups of $f^{\circ n}(x)$ embed into the $n$-fold iterated wreath product of $\Gal(f)$, the Galois group of $f$ \cite{odoniPlms}. More recently, \cite{newRed} showed that there are finitely many irreducible quadratic polynomials where $f^{\circ n}$ is reducible if $n\ge 2$. In \cite{DanFein} it was already shown that there are infinitely many irreducible quadratics $f$ with $f\circ f$ reducible. It should be noted that once a term in the sequence is reducible all subsequent terms will be reducible. Additionally, Hindes \cite{hindes} has shown that the Galois group of $f^{\circ n}$ can fail to be the full $n$-fold iterated wreath product even when $f^{\circ n}$ is always irreducible. 

The phenomena that we focus on in this paper will be called emergent reducibility (or ``newly reducible'' in \cite{newRed}). 
\begin{definition}[Emergent Reducibility at Depth $n$] We say a polynomial $f(x) \in K[x]$ has \textbf{emergent reducibility at depth $n$} if and only if $f^{\circ i}(x)$ is irreducible over $K[x]$ for $0\leq i \leq n-1$ and $f^{\circ n}(x)$ is reducible over $K[x]$. Note that $n$ counts the number of compositions so $f^{\circ 0}(x) = f(x)$.
\end{definition}

We will prove 
\begin{theorem}\label{inf-er}
There are infinitely many irreducible cubics $f\in \mathbb{Z}[x]$ with emergent reducibility at depth 1.
\end{theorem}

To do this, we show that the family 
\[
f_a(x) =  -8ax^3-(8a+2)x^2+(4a-1)x+a
\]
is irreducible for infinitely many integers $a$ and that $f_a\circ f_a(x)$ is the product of a cubic and sextic polynomial for all $a$, namely 
\[
g_a(x) = 32a^2x^3+(32a^2+16a)x^2+(-16a^2+12a+2)x+(-4a^2-4a+1)
\]
and
\begin{equation}
\begin{split}
h_a(x) = \fafaS.
\end{split}
\end{equation}

\section{The Irreducibility of $f_a$}
Since reducible polynomials will have reducible iterates, it is crucial to show that infinitely many polynomials of the form $f_a(x)$ for $a\in \mathbb{Z}$ are irreducible. This can be accomplished, somewhat unsatisfactorily, by use of Hilbert's Irreducibility Theorem. We can consider our parameterized family as a family of polynomials in two variables, $f_a(x) = f(a,x) \in \mathbb{Q}[a,x]$. Since $f(a,x)$ is linear in $a$ and the ``coefficients'' have no common factors in $\mathbb{Q}[x]$, we know that $f(a,x)$ is irreducible in $\mathbb{Q}[a,x]$. Hilbert's Irreducibility Theorem\cite{serre} ensures that for infinitely many values of $a \in \mathbb{Z}$ the specialization $f_a(x)$ is irreducible in $\mathbb{Q}[x]$. Unfortunately, this fails to give any specific values of $a$ such that $f_a(x)$ is irreducible.

The following theorem gives a more effective determination of when $f_a(x)$ is irreducible. We note that computational evidence suggests that $f_a(x)$ is irreducible for every non-zero $a$ (verified for $0<|a|\leq 10^6$) but for our purposes the following is sufficient.

\begin{theorem}\label{irrFa}
For $a \in \mathbb{Z}$ with $3\nmid a$, the polynomial $f_a(x) = \fa$ is irreducible over $\mathbb{Z}/3$ and therefore irreducible in $\mathbb{Z}[x]$.
\end{theorem}

\begin{proof} 
Since $3\nmid a$, we have
\begin{align}
f_a(x) &= \fa \notag\\
 &\equiv \begin{cases}
x^3+2x^2+1 &\text{if $a\equiv 1\mod(3)$}\\
x^3+2x+1 &\text{if $a\equiv 2\mod(3)$.}
\end{cases}
\end{align} 
Both polynomials are easily seen to be irreducible over $\mathbb{Z}/3$.\\ 
\end{proof}

Considering $a = 3t$ for $t\in \mathbb{Z}$, reduction modulo other small primes shows that even more values of $a$ will result in irreducible polynomials, matching computational evidence. 

\section{The Reducibility of $f_a\circ f_a$}
The required reducibility is easily verified. The composition of $f_a$ with itself is 
\begin{equation}
\begin{split}
f_a\circ f_a(x) = \fafa.
\end{split}
\end{equation}
With the cubic $g_a(x)$ and the sextic $h_a(x)$ defined as
\[
g_a(x) = 32a^2x^3+(32a^2+16a)x^2+(-16a^2+12a+2)x+(-4a^2-4a+1)
\]
and
\begin{equation}
\begin{split}
h_a(x) = \fafaS
\end{split}
\end{equation}
then
\begin{equation}
\begin{split}
g_a(x)h_a(x) = f_a\circ f_a(x).
\end{split}
\end{equation}
Thus, $f_a \circ f_a$ is reducible and together with \ref{irrFa} we have proven:
\begin{theorem}
For $3\nmid a \in \mathbb{Z}$, the polynomials $f_a(x) = -8ax^3-(8a+2)x^2+(4a-1)x+a$ are irreducible with emergent reducibility of depth $1$.
\end{theorem}

Theorem \ref{inf-er} is a direct consequence of the existence of the family $f_a(x)$. We also note that computational evidence suggests $g_a$ and $h_a$ are irreducible over $\mathbb{Z}[x]$, but this in ancillary to our requirements. In the next section we show that there are a number of other examples of this behavior.

\section{Other Examples}
There are many examples of non-monic cubics with depth 1 emergent reducibility, including other parameterizable families. The more interesting situation is the case of monic integral cubics where there seems to be only a finite number with depth one emergent reducibility. The following is a list of all known examples where the absolute value of the coefficients are less than $500$.
\begin{align}
x^3 &\pm 9x^2+23x\pm 13\\
x^3&\pm 6x^2 +11x \pm 5\\
x^3&\pm x^2-3x\mp 1\\
x^3&\pm 4x^2 +3x \mp 1
\end{align}

This leads us to make the following conjecture:
\begin{conj}
There are only finitely many monic cubics in $\mathbb{Z}[x]$ with depth one emergent reducibility.
\end{conj}

\bibliographystyle{amsplain}
\bibliography{cubic-er-d1}
\end{document}